\documentclass[8pt]{article}
\usepackage{amssymb,amsmath,color}
\usepackage{amsthm}
\usepackage{fullpage}
\usepackage{mathrsfs}
\usepackage{units}
\allowdisplaybreaks
\usepackage{color}

\theoremstyle{definition}

\let\oldsection\section
\renewcommand\section{\setcounter{equation}{0}\oldsection}

\newcommand{\be}{\begin{equation}\label}
\newcommand{\ee}{\end{equation}}
\newcommand{\beaa}{\begin{eqnarray}}

\newcommand{\eea}{\end{eqnarray}}
\newcommand{\nn}{\nonumber}
\newcommand{\Om}{\Omega}

\newcommand{\pa}{\partial}

\newcommand{\si}{\sigma}

\newcommand{\Lin}{L^\infty(\Om)}

\newcommand{\eps}{\varepsilon}
\newcommand{\al}{\alpha}

\newcommand{\into}{\int_{\Omega}}

\newcommand{\Tmax}{T_{max}}

\newcommand{\norm}[2][ ]{\|#2\|_{#1}}

\newcommand{\R}{\mathbb{R}}

\newcommand{\io}{\int_\Om}

\newcommand{\fat}{\qquad \mbox{for all }t\in(0,\Tmax)}
\newcommand{\fas}{\qquad \mbox{for all }s\in(0,\Tmax)}

\theoremstyle{plain}
\newtheorem{thm}{Theorem}[section]
\newtheorem{lem}[thm]{Lemma}
\theoremstyle{definition}
\theoremstyle{remark}
\allowdisplaybreaks

\title{Superlinear transmission in an indirect signal production chemotaxis system}
\author{ Xinru Cao \footnote{xcao@dhu.edu.cn, North Renmin Road 2999, 201620 Shanghai, China}\\
\small{ School of Mathematics and Statistics},\\
\small{ Donghua University}
}
\date{}
\begin{document}
\maketitle
\begin{abstract}
\noindent 
In this paper, the indirect signal production system with nonlinear transmission is considered
\[
\left\{
\begin{array}{lll}
& u_t = \Delta u-\nabla\cdot(u \nabla v), \\
\displaystyle
& v_t  =\Delta v-v+w,\\
\displaystyle
& w_t =\Delta w-w+ f(u)
\end{array}
\right.
\]
in a bounded smooth domain $\Omega\subset \R^n$ associated with homogenous Neumann boundary conditions{,  where} $f\in C^1([0,\infty))$ satisfies $0\le f(s) \le s^{\alpha}$ with {$\alpha>0$}. It is known from \cite{tao2024non_trans} that the system possesses a global bounded solution if {$0<\al<\frac 4n$} when $n\ge 4$. 
In the case $n\le 3$ and {if we} consider superlinear  transmission, no regularity of $w$ or $v$ can be derived directly. In this work, we show that if {$0<\al< \min\{\frac 4n,1+\frac 2n\}$}, the solution is global and bounded via an approach based on the maximal Sobolev regularity. 
{\bf Math Subject Classification (2010):} 35K55, 35B33, 92C17
\end{abstract}

\section{Introduction}\label{introduction}
The pioneer chemotaxis model was established by Keller and Segel in 1970's \cite{keller1971model}, to describe cells' motion which biased by a chemical substance secreted by the cells themselves. 

Recently, phenotype heterogeneity is assumed among the cells population, they are divided into two groups; one group secrets chemical signal, the other group is attracted by the signal. These abilities do not hold simultaneously, some kind of trade-off is observed between two groups. This model can be seen as an extension of the classical Keller-Segel model, with $u$ and $w$ presenting the density of ``chemotaxis" and ``secreting" groups, respectively, and $v$ is the concentration of chemical substance. Under these assumptions, the following system is proposed \cite{macfarlane2022impact}:
\begin{equation}
\label{eq}
\left\{
\begin{array}{lllll}
& u_t = \Delta u-\nabla\cdot(u \nabla v)-k_1 u+k_2 w,  &(x,t)\in \Omega\times(0,T),\\[2pt]
\displaystyle
& v_t  =\Delta v-v+w, &(x,t)\in \Omega\times(0,T),\\[2pt]
\displaystyle
& w_t =\Delta w-w+ f(u), &(x,t)\in \Omega\times(0,T),\\[2pt]
& \partial_{\nu} u= \partial_{\nu} v=\partial_{\nu} w=0, &(x,t)\in \par \Omega\times(0,T),\\[2pt]
 &u(x,0)=u_0(x), v(x,0)=v_0(x), w(x,0)=w_0(x), & x\in\Omega,
\end{array}
\right.
\end{equation}
where $\Omega\subset \R^n$ ($ n \ge 1$) is smooth and bounded, {$T>0$}, $\nu$ denotes the outer normal vector on $\pa\Omega$, $k_1, k_2\ge 0${,} and
 \begin{align}
\label{con:f}
&
	f\in C^1([0,\infty)) \text{ and } 0\le f(s)\le s^{\al}  \text{ with some } {\al>0},\\
	\label{con:initial}
	& u_0\in C^0(\overline{\Om}), v_0\in W^{2,\infty}(\Om), w_0\in W^{2,\infty}(\Omega) \text{ are nonnegative}.	
\end{align}
{Compared} to the classical Keller-Segel system, the indirect signal production mechanism in \eqref{eq}
 manifests a stronger tendency toward stabilization; it is proven {in \cite{fujie2017application,painter2023switching}} that in the case $f(u)\equiv u$, the solution of \eqref{eq} is bounded for $n\le 3$ , which is in contrast to the classical Keller-Segel system that there exist finite time blowup solutions in two and three dimensions \cite{WinklerFinitetimeBlowup2013}. 
 Considering {\em direct} nonlinear signal production model:
 \begin{equation}
 \label{eq:1}
 u_t=\Delta u-\nabla\cdot(u\nabla v), \quad v_t=\Delta v-v+f(u),
  \end{equation}
 where $f$ satisfies \eqref{con:f}, the value $\alpha$ determines boundedness or blow up of the solution; if $0<\alpha<\frac 2n$, the solution is global and bounded \cite{liu2016non_signal}; if $f(s)= s^{\alpha}$ with $\alpha>\frac 2n$ and the second equation is replaced by 
 $\Delta v-\frac{1}{|\Omega|}\into f(u(\cdot,t))+f(u)=0$ in \eqref{eq:1}, finite time blowup solution is constructed in radial symmetric setting \cite{winkler2018critical}.
 In the present work, we study {\em indirect} nonlinear signal production mechanism in \eqref{eq}, and obtain the following:
\begin{thm}\label{thm}
Let $ n \le 3$,  $k_1=k_2=0$, $\Omega\subset\R^n$ be bounded domain with smooth boundary, {$f$ satisfy \eqref{con:f}} and {let} $(u_0,v_0,w_0)$ comply with \eqref{con:initial}. If {$0<\al< \min\{\frac4n,1+\frac 2n\}$}, then the system \eqref{eq} admits a {unique} global classical solution $(u,v,w)$ satisfying
$$
(u,v,w)\in (C^0(\overline{\Om}\times[0,\infty))\cap 
C^{2,1}(\overline{\Om}\times(0,\infty)))^3, 
$$
which is  uniformly bounded.
\end{thm}
Noting that the condition on $\alpha$ will allow superlinear transmission in the third equation, no integrable regularity is known for the production term $f(u)$, we can not gain any information about $w$. The strategy of the proof is to track the evolution of $\into u^p(\cdot,t)$, and run a loop through each equation via the maximal Sobolev regularity.
\section{Preliminaries }
\label{sec:pre}

Let us first give local existence theory for \eqref{eq}. It can be proven by fix point argument and Schauder theory.
\begin{lem}\label{lem:locexist}
Assume that $\Om\subset\R^n$ $(n\ge 1)$ is a bounded domain with smooth boundary, {$f$ satisfies \eqref{con:f} and }that the initial data $(u_0,v_0,w_0)$
satisfy \eqref{con:initial}. 
There exists $\Tmax\in(0,\infty]$ 
with the property that the problem possess a unique classical 
solution $(u,v,w)$ 
satisfy 
$$
(u,v,w)\in (C^0(\overline{\Om}\times[0,\Tmax))\cap 
C^{2 ,1}(\overline{\Om}\times(0,\Tmax)))^3,  {\text{ and }  u,v,w>0  \text{ in } \overline{\Omega}\times(0,\Tmax).}
$$
Moreover, if $\Tmax<\infty$, then
$
\norm[\Lin]{u(\cdot,t)}\to 
\infty, 
\text { as }t\nearrow \Tmax. 
$
\end{lem}

The next lemma can be proven by simply integrating over $\Omega$. The second result is of particular interest for the superlinear transmission case, that no regularity for $w$ is known, however, its $L^1(\Omega)$-norm can be estimated by $u$ non-locally in time.
\begin{lem}\label{lem:wL1_u}
Let $s\in(0,\Tmax)$. {There hold that}
	\begin{align}
	\label{eq:lem:uL1}
	& \norm[L^1(\Omega)]{u(\cdot,s)} =\norm[L^1(\Omega)]{u_0},\\
	\label{eq:lem:wL1_u}
		&\norm[L^1(\Omega)]{ w(\cdot,s)} \le {\int_0^s e^{-(s-\sigma)} \norm[L^1(\Om)]{u^{\al}(\cdot,\sigma)} d \sigma +\norm[L^1(\Omega)]{w_0} }\fas.
	\end{align}
\end{lem}

The following result on maximal Sobolev regularity is the key tool for us to treat the term with space-time integrations. The proof can be found in \cite[Lemma 2.3]{cao_tao2021}.
\begin{lem}
\label{lem:maximal}
For any given $q,r\in(1,\infty)$ and for any fixed $\kappa\in(0,1)$,
there exists $C=C(q,r,\kappa)>0$ with the
following property: \\
For all $T>0$, let $h\in C^0(\overline{\Omega}\times(0,T))$ and let
$z$ be a classical solution of
  \begin{equation}
        \left\{ \begin{array}{lll}
     &z_t =\Delta z-z +h,
    \qquad & (x, t)\in \Omega\times (0, T), \\[1mm]
   & \partial_{\nu} z =0,
    \qquad & (x, t)\in \partial\Omega\times (0, T), \\[1mm]
      &z(x,0)=z_0(x),     \qquad & x\in\Omega,
    \end{array} \right.
\end{equation}
 then we have
  \begin{align}
  \int_0^t e^{-\kappa r(t-s)}\|\Delta z(\cdot,s)\|_{L^q(\Omega)}^r ds
 \le C\int_0^t e^{-\kappa r(t-s) }\|h(\cdot,s)\|_{L^q(\Omega)}^r ds
     +C\|z_0\|_{W^{2, q}(\Omega)}^r
  \end{align}
  for all $t\in (0, T)$.
\end{lem}

As the final preparation, we present a fundamental estimate to treat a term with double integration in time.
\begin{lem}
\label{lem:fubini}
	Let $\delta\in(0,1)$, $q\ge 1$, $T>0$ and let $ g \in  L^\infty((0,T))$ be nonnegative. {It holds that}
	\begin{align}\nn
		&~~\int_0^t e^{-\delta(t-s)}\left(\int_0^s e^{-(s-\sigma)} g(\si ) d\sigma \right)^{q} ds\le \frac{1}{1-\delta}\int_0^t e^{-\delta(t-s)} g^q(s) ds  \text{ for all } t\in(0,T). \end{align}
\end{lem}

\begin{proof}
	First we apply { H\"older's} inequality to find that
	\begin{align*}
		\left(\int_0^s e^{-(s-\si )} g(\si ) d\si \right)^q\le \left(\int_0^s e^{-(s-\si)} g^q(\si ) d\si \right) \left(\int_0^s e^{-(s-\si)} d\si\right)^{q(1-\frac 1q)}\le \int_0^s e^{-(s-\si)} g^q(\si ) d\si .
	\end{align*}
	Thanks to the Fubini Theorem, we have
	\begin{align*}
	\nn
&~~~~\int_0^t e^{-\delta(t-s)}\left(\int_0^s e^{-(s-\sigma)} g(\si ) d\sigma \right)^{q} ds\\
		&\le \int_0^t e^{-\delta(t-s)} \int_0^s e^{-(s-\si)} g^q(\si ) d\si ds = \int_0^t e^{-\delta t+\si} g^q(\si )\int_\si^t  e^{-(1-\delta)s} ds d\si \le \frac{1}{1-\delta} \int_0^t e^{-\delta(t-\si )} g^q(\si) d\si .
	\end{align*}
\end{proof}



\section{Proof of Theorem \ref{thm}}
 In order to show boundedness of $\norm[L^p(\Omega)]{u}$ for arbitrarily large $p$, we need to fix parameters properly, {the choices of the parameters} are guaranteed by the condition on $\alpha$. 
\begin{lem}
	\label{lem:parameters}
	Let $n\le 3$. Suppose that $1<\al<\min\{\frac 4n,1+\frac 2n\}$. We can fix 
	\begin{align}
		\label{para:tau}
		&\tau>1 \text{ satisfying } \al-\frac 2n<\frac{1}{\tau}<\min\left\{1,\frac 2n\right\}, \\
		\label{para:theta}
		& \theta'>1 \text{ be such that } \frac{1}{\tau}-\min\left\{1,\frac 2n\right\}<\frac{1}{\theta'}< \frac{1}{\tau}+\frac 2n-\al,  \text{ and } \quad\frac{1}{\theta}+\frac{1}{\theta'}=1.\\
		\label{para:p}
		\text{ {Let} } & b:=\frac{n-\frac{n}{\theta'}}{2-\frac{n}{\tau}+n} \in(0,1), \quad
			 p> \max\left\{1,\frac{\al(n-2)_+}{n}, \frac{\alpha\tau (n-2)_+}{n},\frac{\frac 2n-\frac{1}{\theta'}}{1-b}+1-\frac 2n\right\},\\
		\label{para:a_b_lambda}
		& a:=\frac{\frac{np}{2}-\frac{n}{2\theta}}{1-\frac{n}{2}+\frac{np}{2}} \in(0,1), \quad \lambda=2a\in(0,2),\quad  \\
		\label{para:eta}
		\text{ and that }
			&\eta_1:=\frac{n\frac{2}{2-\lambda}b(\al-\frac{1}{\tau})}{2-n+np}>0, \quad \eta_2:=\frac{n\frac{2}{2-\lambda}(1-b)(\al-1)}{2-n+np}>0, \quad \eta_3:=\frac{n\frac{2}{2-\lambda}(\al-1)}{2-n+np}\in(0,1).
		\end{align}
	We can find $\beta,\beta'\in(1,\infty)$ satisfying $ \frac{1}{\beta}+\frac{1}{\beta'}=1$, and that
\begin{align}
\label{para:<1}
0<\eta_1+\eta_2<1,
	\quad
\beta\eta_1<1, \quad \beta'\eta_2<1.
\end{align}
\end{lem}

\begin{proof}
	First $n\le 3$ and $1<\al<\min\left\{\frac 4n,1+\frac2n\right\}$ {yields} that $\al-\frac 2n<\min\left\{1,\frac 2n\right\}$. Hence we can fix $\tau$ and $\theta'$ as in \eqref{para:tau} and (\ref{para:theta}). Moreover, we see that $\frac{1}{\theta'}<\frac{2}{n}$, and $\frac{1}{\theta}>(1-\frac 2n)_+$, which implies that $a\in(0,1)$, as well as $\lambda\in(0,2)$. Noting that (\ref{para:theta}) ensures 
	\begin{align}
	\label{eq:lem:para:1}
\al-\frac2n<\frac{1}{\tau}-\frac{1}{\theta'}<\min\left\{1,\frac 2n\right\}.
\end{align} 
Hence $b\in(0,1)$. 
And \eqref{eq:lem:para:1} also implies $\al<1+\frac 2n-\frac 1{\theta'}$ because {of} $\tau>1$, then $\eta_3\in(0,1)$. {Furthermore},  we have $0<\al-\frac{1}{\tau}<\frac2n-\frac1{\theta'}$,  meaning that
\begin{align*}
\nn
\eta_1+\eta_2
	 &=\frac{(1-\frac 1{\theta'})(\al-\frac 1\tau)}{(\frac{2}{n}-\frac{1}{\theta'})(1-\frac{1}{\tau}+\frac 2n)} +\frac{(\frac 2n-\frac 1\tau+\frac{1}{\theta'})(\al-1)}{(\frac{2}{n}-\frac{1}{\theta'})(1-\frac{1}{\tau}+\frac 2n)}
	 <\frac{\alpha-\frac{1}{\tau}}{\frac 2n-\frac{1}{\theta'}}<1.
\end{align*}
Therefore, we can fix $\beta \in  (\frac{1}{1-\eta_2},\frac{1}{\eta_1})$. It is convenient to check the inequalities in (\ref{para:<1}) hold.
\end{proof}

Via standard $L^p$ testing procedure, our main idea is to control the cross-diffusion contributions toward dissipation of $u$. Let us first prepare the following spatio-temporal estimate by simple interpolation.
\begin{lem}
\label{lem:u_nablau}
Let $\al>0$,  $\mu>0$, $\delta>0$ and $\frac{1}{\alpha}<k< \frac{np}{\alpha(n-2)_+}$ satisfy $ n\mu(\al-\frac 1k)\le np-n+2$. There exists $C>0$ such that for all $t\in(0,\Tmax)$,
	\begin{align}\label{eq:lem:spa_temp}
		\int_0^t e^{-\delta(t-s)} \norm[L^k(\Omega)]{u^\al(\cdot,s)}^\mu ds\le C\left(\int_0^t e^{-\delta(t-s)} \norm[L^2(\Omega)]{\nabla u^{\frac p2}(\cdot,s)}^2 ds \right)^{\frac{n\mu(\al-\frac 1k)}{np-n+2}}+C.
	\end{align}
\end{lem}
\begin{proof}
Let $d=\frac{\frac{np}{2}-\frac{np}{2\al k}}{1-\frac n2+\frac{np}{2}}$. We can check that $d\in (0,1)$.
The Gargliardo-Nirenberg inequality implies that 
	\begin{align}
	\nn
	\norm[L^k(\Omega)]{u^{\alpha}}^{\mu} &=\norm[L^{\frac{2\al k}{p}}(\Om)]{u^{\frac{p}{2}}}^{\frac{2\al\mu}{p}}\le c_1\norm[L^2(\Om)]{\nabla u^{\frac{p}{2}}}^{\frac{2\al\mu}{p}d}\norm[L^{\frac{2}{p}}(\Om)]{u^{\frac{p}{2}}}^{\frac{2\al\mu}{p}(1-d)}
	+ c_1 \norm[L^{\frac{2}{p}}(\Om)]{u^{\frac{p}{2}}}^{\frac{2\al\mu}{p}}\\
	\nn
	&\le  c_2 \norm[L^2(\Om)]{\nabla u^{\frac{p}{2}}}^{\frac{2\al\mu}{p}d}+c_2 \fat.
	\end{align}
	Integrating over $(0,t)$, we obtain that
\begin{align}
\nn
\int_0^t e^{-\delta(t-s)} \norm[L^k(\Omega)]{u^\al(\cdot,s)}^\mu ds &\le c_2\int_0^t e^{-\delta(t-s)}
\norm[L^2(\Om)]{\nabla u^{\frac{p}{2}}(\cdot,s)}^{\frac{2\al\mu}{p}d} ds +\frac{c_2}{\delta}.
\end{align}
If $n\mu(\al-\frac 1k)= np-n+2$, this is \eqref{eq:lem:spa_temp}. If $n\mu(\al-\frac 1k)< np-n+2$, meaning that 
$\frac{2\al\mu}{p}d<2$. Applying {H\"older}'s inequality, we see that
\begin{align}
\nn
\int_0^t e^{-\delta(t-s)}
\norm[L^2(\Om)]{\nabla u^{\frac{p}{2}}(\cdot,s)}^{\frac{2\al\mu}{p}d} ds &\le \left(\int_0^t e^{-\delta(t-s)}
 \norm[L^2(\Om)]{\nabla u^{\frac{p}{2}}(\cdot,s)}^2
ds \right)^{\frac{\al\mu}{p}d} \left(\int_0^t e^{-\delta(t-s)} ds \right)^{1-\frac{\al\mu}{p}d}\\
&\le \left(\frac{1}{\delta}\right)^{1-\frac{\al\mu}{p}d}  \left(\int_0^t e^{-\delta(t-s)}
 \norm[L^2(\Om)]{\nabla u^{\frac{p}{2}}(\cdot,s)}^2
ds \right)^{\frac{\al\mu}{p}d}.
\end{align}
\end{proof}
Our next step to handle the cross-diffusion effect involves a priori estimate of $w$. In particular, we focus on the superlinear transmission case, this causes poor information on the source term in the third equation, hence no regularity is known for $w$. However, we can still use interpolation and control spatial-temporal  integration against $u^\alpha$.
\begin{lem}\label{lem:w_u}
Let $\delta\in (0,\min\{1,\frac{2}{2-\lambda},\frac{2}{2-\lambda}b\beta\})$.
 There {exists} $C>0$ such that for all $t\in(0,\Tmax)$, we have
	\begin{align}
	\nn
	&~~~~\int_0^t e^{-\delta(t-s)}\norm[L^{\theta'}(\Omega)]{w(\cdot,s)}^{\frac{2}{2-\lambda}} ds \\
	\nn
&\le C \left( \int_0^t e^{-\delta(t-s)} \norm[L^\tau(\Omega)]{u^\al (\cdot,s)}^{\frac{2}{2-\lambda} b \beta} ds \right)^{\frac{1}{\beta}} 
\left( \int_0^t e^{-\delta(t-s)} \norm[L^1(\Omega)]{u^\al(\cdot,s)}^{\frac{2}{2-\lambda}(1-b) \beta'} ds
\right)^{\frac{1}{\beta'}}\\
\nn
&\quad  + C \left( \int_0^t e^{-\delta(t-s)} \norm[L^\tau(\Omega)]{u^\al (\cdot,s)}^{\frac{2}{2-\lambda} b \beta} ds\right)^{\frac{1}{\beta}} 
+C\left( \int_0^t e^{-\delta(t-s)} \norm[L^1(\Omega)]{u^\al(\cdot,s)}^{\frac{2}{2-\lambda}(1-b) \beta'} ds
\right)^{\frac{1}{\beta'}}\\
&\qquad
+ C\int_0^t e^{-\delta(t-s )} \norm[L^1(\Om)]{u^{\al}(\cdot,s)}^{\frac{2}{2-\lambda}} ds +C  .
	\end{align}
\end{lem}
\begin{proof}
{Invoking the Gargliardo-Nirenberg inequality, we obtain $c_1>0$ such that}
	\begin{align*}
		\norm[L^{\theta'}(\Omega)]{w(\cdot,s)}\le c_1 \norm[L^\tau(\Omega)]{\Delta w(\cdot,s)}^b\norm[L^1(\Omega)]{w(\cdot,s)}^{1-b}+c_1 \norm[L^1(\Omega)]{w(\cdot,s)}  \text{ for all } s\in(0,\Tmax).
	\end{align*}
Multiplying $e^{-\delta(t-s)}$ and integrating over $(0,t)$, applying Young's inequality,  we get
\begin{align}
\label{eq:lem:w_u:1}
\nn
		&~~~~\int_0^t e^{-\delta(t-s)}\norm[L^{\theta'}(\Omega)]{w(\cdot,s)}^{\frac{2}{2-\lambda}} ds  \\
		\nn
		&\le c_1  \int_0^t e^{-\delta(t-s)} \norm[L^\tau(\Omega)]{\Delta w(\cdot,s)}^{\frac{2}{2-\lambda} b}\norm[L^1(\Omega)]{w(\cdot,s)}^{\frac{2}{2-\lambda}(1-b)} ds
		     +c_1\int_0^t e^{-\delta(t-s)}\norm[L^1(\Omega)]{w(\cdot,s)}^{\frac{2}{2-\lambda}} ds\\
\nn
   &\le c_1 \left(  \int_0^t e^{-\delta(t-s)} \norm[L^\tau(\Omega)]{\Delta w(\cdot,s)}^{\frac{2}{2-\lambda} b \beta } ds \right)^{\frac{1}{\beta}} 
   \left( \int_0^t e^{-\delta(t-s)} \norm[L^1(\Omega)]{w(\cdot,s)}^{\frac{2}{2-\lambda}(1-b)\beta'} ds\right)^{\frac{1}{\beta'}}\\
   &\quad\quad  +c_1\int_0^t e^{-\delta(t-s)}\norm[L^1(\Omega)]{w(\cdot,s)}^{\frac{2}{2-\lambda}} ds  \fat.
		\end{align}
	Since $\frac{2}{2-\lambda}(1-b)\beta'>1$ due to \eqref{para:p}, employing Lemmata \ref{lem:wL1_u} and \ref{lem:fubini}, we {find} $c_2>0$ such that
\begin{align}
\label{eq:lem:w_u:2}
\nn
&~~~~\int_0^t e^{-\delta(t-s)} \norm[L^1(\Omega)]{w(\cdot,s)}^{\frac{2}{2-\lambda}(1-b)\beta'} ds \\
\nn
&\le \int_0^t e^{-\delta(t-s)} \left( { \int_0^s e^{-(s-\sigma)} \norm[L^1(\Om)]{u^{\al}(\cdot,\sigma)} +\norm[L^1(\Omega)]{w_0}  d \sigma }\right) ^{\frac{2}{2-\lambda}(1-b)\beta'} ds\\
\nn
&\le  c_2\int_0^t e^{-\delta(t-s)} \left( \int_0^s e^{-(s-\sigma)} \norm[L^1(\Om)]{u^{\al}(\cdot,\sigma)} d \sigma \right) ^{\frac{2}{2-\lambda}(1-b)\beta'} ds + c_2 \int_0^t e^{-\delta(t-s)}  \norm[L^1(\Omega)]{w_0}^{\frac{2}{2-\lambda}(1-b)\beta'}ds \\
&\le \frac{c_2}{1-\delta}  \int_0^t e^{-\delta (t-s)} \norm[L^1(\Om)]{u^{\al}(\cdot,s)}^{\frac{2}{2-\lambda}(1-b)\beta'} ds + \frac{c_2}{\delta}\norm[L^1(\Omega)]{w_0}^{\frac{2}{2-\lambda}(1-b)\beta'}.
\end{align}
{Due to $(x+y)^{\frac{1}{\beta'}}\le x^{\frac{1}{\beta'}}+y^{\frac{1}{\beta'}}$ for $x,y\ge 0$ with $\beta'>1$, we have} 
\begin{align}
\label{eq:lem:w_u:3}
\nn
&~~~~ \left( \int_0^t e^{-\delta(t-s)} \norm[L^1(\Omega)]{w(\cdot,s)}^{\frac{2}{2-\lambda}(1-b)\beta'} ds\right)^{\frac{1}{\beta'}}\\
& \le \left( \frac{c_2}{1-\delta}\right)^{\frac{1}{\beta'}}  \left(\int_0^t e^{-\delta (t-s)} \norm[L^1(\Om)]{u^{\al}(\cdot,s)}^{\frac{2}{2-\lambda}(1-b)\beta'} ds  \right)^{\frac{1}{\beta'}}
 +\left(\frac{c_2}{\delta}\right)^{\frac{1}{\beta'}}\norm[L^1(\Omega)]{w_0}^{\frac{2}{2-\lambda}(1-b)}.
\end{align}
Similar to \eqref{eq:lem:w_u:2}, since $\frac{2}{2-\lambda}>1$, Lemma \ref{lem:fubini} is applicable to yield that
	\begin{align}
	\label{eq:lem:w_u:4}
	\int_0^t e^{-\delta(t-s)}\norm[L^1(\Omega)]{w(\cdot,s)}^{\frac{2}{2-\lambda}} ds \le \frac{c_3}{1-\delta} \int_0^t e^{-\delta(t-s)} \norm[L^1(\Om)]{u^{\al}(\cdot,s)}^{\frac{2}{2-\lambda}} ds +\frac{c_3}{\delta} \norm[L^1(\Omega)]{w_0}^{\frac{2}{2-\lambda}}.
	\end{align}
	An application of Lemma \ref{lem:maximal} implies $c_4>0$ such that
	\begin{align}
	\label{eq:lem:w_u:5}
 \int_0^t e^{-\delta(t-s)} \norm[L^\tau(\Omega)]{\Delta w(\cdot,s)}^{\frac{2}{2-\lambda} b \beta } ds \le c_4 	 \int_0^t e^{-\delta(t-s)} \norm[L^\tau(\Omega)]{ u^\al(\cdot,s)}^{\frac{2}{2-\lambda} b \beta } ds+c_4
 \norm[W^{2,\infty}(\Omega)]{w_0}^{\frac{2}{2-\lambda} b \beta}.
	\end{align}
Substituting \eqref{eq:lem:w_u:3},\eqref{eq:lem:w_u:4} and \eqref{eq:lem:w_u:5} into \eqref{eq:lem:w_u:1}, {the assertion follows}.		
\end{proof}
 
Combining Lemmata \ref{lem:w_u} and \ref{lem:u_nablau}, we can readily control cross-diffusion by dissipation of $u$.
\begin{lem}\label{lem:laplacianv_nablau}
Let $\delta\in (0,\min\{1,\frac{2}{2-\lambda},\frac{2}{2-\lambda}b\beta\})$.
For any $\eps>0$, there is $C>0$ such that
\begin{align}
\int_0^t e^{-\delta(t-s)} \norm[L^{\theta'}(\Om)]{\Delta v(\cdot,s)}^{\frac{2}{2-\lambda}} ds \le  \eps \int_0^t e^{-\delta(t-s)} \norm[L^2(\Om)]{\nabla u^{\frac p2}(\cdot,s)}^2 ds +C \fat.
\end{align}
\end{lem}
\begin{proof}
Thanks to \eqref{para:p}, it holds that $\frac{1}{\al}<\tau< \frac{np}{\alpha(n-2)_+}$ and $\frac{1}{\al}<1< \frac{np}{\alpha(n-2)_+}$, also \eqref{para:<1} ensures $ \frac{n (\frac{2}{2-\lambda}b\beta)(\al-\frac 1\tau)}{ np-n+2}=\beta\eta_1< 1$ and 
$\frac{n\mu (\frac{2}{2-\lambda})(1-b)\beta'(\al-1)}{ np-n+2}=\beta'\eta_2< 1$, together with the fact that $\eta_3<1$ from \eqref{para:eta}, we invoke Lemma \ref{lem:u_nablau} to derive
\begin{align*}
&\int_0^t e^{-\delta(t-s)} \norm[L^\tau(\Omega)]{u^\al (\cdot,s)}^{\frac{2}{2-\lambda} b \beta} ds
\le c_1 \left( \int_0^t e^{-\delta(t-s)} \norm[L^2(\Om)]{\nabla u^{\frac p2}(\cdot,s)}^2 ds \right)^{\beta\eta_1}+c_1,\\
& \int_0^t e^{-\delta(t-s)} \norm[L^1(\Omega)]{u^\al(\cdot,s)}^{\frac{2}{2-\lambda}(1-b) \beta'} ds
\le c_2 \left( \int_0^t e^{-\delta(t-s)} \norm[L^2(\Om)]{\nabla u^{\frac p2}(\cdot,s)}^2 ds \right)^{\beta'\eta_2} +c_2,\\
\text{ and } 
&\int_0^t e^{-\delta(t-s )} \norm[L^1(\Om)]{u^{\al}(\cdot,s)}^{\frac{2}{2-\lambda}} {ds}\le  c_3\left( \int_0^t e^{-\delta(t-s)} \norm[L^2(\Om)]{\nabla u^{\frac p2}(\cdot,s)}^2 ds \right)^{\eta_3} +c_3.
\end{align*}
Applying Lemmata \ref{lem:maximal} and \ref{lem:w_u}, we obtain $c_4, c_5>0$ such that
\begin{align}
\nn
&~~~~\int_0^t e^{-\delta(t-s)} \norm[L^{\theta'}(\Om)]{\Delta v(\cdot,s)}^{\frac{2}{2-\lambda}} ds \\
\nn
&\le c_4\int_0^t e^{-\delta(t-s)} \norm[L^{\theta'}(\Om)]{w(\cdot,s)}^{\frac{2}{2-\lambda}} ds+c_4\norm[W^{2,\infty}(\Omega)]{v_0}^{\frac{2}{2-\lambda}}\\
\nn
&\le c_5 \left( \int_0^t e^{-\delta(t-s)} \norm[L^2(\Om)]{\nabla u^{\frac p2}(\cdot,s)}^2 ds \right)^{\eta_1+\eta_2}+c_5 \left( \int_0^t e^{-\delta(t-s)} \norm[L^2(\Om)]{\nabla u^{\frac p2}(\cdot,s)}^2 ds \right)^{\eta_1}\\
&\qquad +c_5 \left( \int_0^t e^{-\delta(t-s)} \norm[L^2(\Om)]{\nabla u^{\frac p2}(\cdot,s)}^2 ds \right)^{\eta_2} +c_5\left( \int_0^t e^{-\delta(t-s)} \norm[L^2(\Om)]{\nabla u^{\frac p2}(\cdot,s)}^2 ds \right)^{\eta_3}+ c_5.\end{align}
Since $0<\eta_1+\eta_2<1$ and $0<\eta_i<1$ for $i=1,2,3$ have been confirmed in Lemma \ref{lem:parameters}, we conclude the assertion by applying Young's inequality.
\end{proof}

We are now prepared to close the loop by employing Lemma \ref{lem:laplacianv_nablau} together with an $L^p$ testing argument with sufficiently large $p$.
\begin{lem}\label{lem:u_Lp}
	Let $n\le 3$. {Assume that} $1<\al<\min\{\frac 4n,1+\frac 2n\}$. For any 
	\[ p> \max\left\{1,\frac{\al(n-2)_+}{n}, \frac{\alpha\tau (n-2)_+}{n},\frac{\frac 2n-\frac{1}{\theta'}}{1-b}+1-\frac 2n\right\},\]
	there is $C>0$ such that 
	\begin{align}
	\into u^p(\cdot,t)\le C  \fat.
	\end{align}
\end{lem}

\begin{proof}
We test the first equation in \eqref{eq} with $pu^{p-1}$, integrate by parts, 
	\begin{align}
	\label{eq:lem:lp:1}
		\frac{d}{dt}\into u^p+\frac{4(p-1)}{p}\into |\nabla u^{\frac p2}|^2= -(p-1)\into u^p\Delta v\le (p-1)\left(\into u^{p\theta}\right)^{\frac{1}{\theta}} \norm[L^{\theta'}(\Om)]{\Delta v} \fat.
	\end{align}
	The Gagliardo-Nirenberg inequality
yields the existence of $c_1>0$ such that
\begin{align*}\nn
 (p-1)\left(\io u^{p\theta}\right)^\frac {1}{\theta}
 &=(p-1)\norm[L^{2\theta}(\Om)]{u^\frac{p}{2}}^2
 \le c_1\norm[L^2(\Om)]{\nabla u^{\frac p2}}^{2a}
 \norm[L^{\frac 2p } (\Om)]{u^\frac p2}^{2(1-a)}
 +c_1\norm[L^{\frac{2}{p}}(\Om)]{u^\frac{p}{2}}^2\\
 &\le c_2 \norm[L^2(\Om)]{\nabla u^{\frac{p}{2}}}^{\lambda}+c_2.
\end{align*}
Therefore, Young's inequality implies existence of $c_3>0$ {fulfilling}
\begin{align}
\label{eq:lem:lp:2}
(p-1)\left(\into u^{p\theta}\right)^{\frac{1}{\theta}} \norm[L^{\theta'}(\Om)]{\Delta v}
\le \frac{(p-1)}{p}\into |\nabla u^{\frac p2}|^2+c_3\norm[L^{\theta'}(\Om)]{\Delta v}^{\frac{2}{2-\lambda}}+c_3.
\end{align}
Fixing $\delta\in (0,\min\{1,\frac{2}{2-\lambda},\frac{2}{2-\lambda}b\beta\})$,
we derive from {the} Gagliardo-Nirenberg inequality to see that
\begin{align}
\label{eq:lem:lp:3}
\delta \into u^p\le \frac{(p-1)}{p}\into |\nabla u^{\frac p2}|^2+c_4.
\end{align}
From \eqref{eq:lem:lp:1}-\eqref{eq:lem:lp:3}, we obtain that
\begin{align}
\frac{d}{dt}\into u^p+\delta\into u^p+\frac{2(p-1)}{p}\into |\nabla u^{\frac{p}{2}}|^2\le c_3\norm[L^{\theta'}(\Om)]{\Delta v}^{\frac{2}{2-\lambda}}
+c_4 \fat.
\end{align}
By applying variation of constants formula and Lemma \ref{lem:maximal} with $\eps=\frac{p-1}{p}>0$, we derive that
\begin{align*}
&~~~~\into u^p(\cdot,t) + \frac{2(p-1)}{p}\int_0^t e^{-\delta(t-s)} \norm[L^2(\Om)]{\nabla u^{\frac p2}(\cdot,s)}^2 ds\\
&\le 
c_3 \int_0^t e^{-\delta(t-s)} \norm[L^{\theta'}(\Om)]{\Delta v(\cdot,s)}^{\frac{2}{2-\lambda}} ds+ \frac{c_4}{\delta} +e^{-\delta t} \into u_0^p \fat.
\end{align*}
{Employing} Lemma \ref{lem:laplacianv_nablau} with $\eps=\frac{p-1}{p}\cdot\frac{1}{c_3}>0$ leads to the assertion.
\end{proof}

\begin{proof}[Proof of Theorem \ref{thm}]
	According to Lemma \ref{lem:u_Lp} and \eqref{eq:lem:uL1}, if $\al<\min\{\frac 4n,1+\frac 2n\}$, then
	\begin{align}
	\sup_{t\in(0,\Tmax)} \norm[L^1(\Omega)]{f(u(\cdot,t ))}<\infty.
	\end{align}
This implies that $\sup_{t\in(0,\Tmax)} \norm[L^\infty(\Omega)]{u(\cdot,t)}<\infty$ by well-established semigroup estimates. Applying Lemma \ref{lem:locexist}, we conclude global existence of the solution and the regularity therein.
\end{proof}

\section*{Acknowledgment} 
The author acknowledges support of the Natural Science Foundation of Shanghai (Project:23ZR1400100).

\def\cprime{$'$}

\end{document}